\documentclass[11pt]{article}
\usepackage{amsmath, amssymb,amsthm, enumitem, babel}
\usepackage[margin=1.0in]{geometry}
\usepackage[title]{appendix}
\usepackage[numbers]{natbib}
\renewcommand{\epsilon}{\varepsilon}
\newcommand{\floor}[1]{\left[ #1 \right]}
\newcommand{\N}{\mathbb{N}}
\newcommand{\Q}{\mathbb{Q}}

\newcommand{\Z}{\mathbb{Z}}
\newcommand{\C}{\mathbb{C}}
\newcommand{\Mod}[1]{\left(\mathrm{mod}\ #1\right)}

\theoremstyle{plain}
\newtheorem{thm}{Theorem}[section]
\newtheorem{lem}[thm]{Lemma}

\newtheorem{cor}[thm]{Corollary}
\theoremstyle{definition}
\newtheorem*{rmk}{Remark}
\newtheorem*{defn}{Definition}
\usepackage{fancyhdr}
\title{\vspace*{-55pt} Fermat quotients and the Ankeny-Artin-Chowla conjecture}
\author{Nic Fellini\\Queen's University\\ \texttt{n.fellini@queensu.ca} 
\and 
M. Ram Murty\footnote{Research partially supported by an NSERC Discovery grant}\\ Queen's University\\
\texttt{murty@queensu.ca}
}
\date{}

\begin{document}

\maketitle
\begin{abstract}
    In this article, we present streamlined proofs of results of Ankeny, Artin, and Chowla concerning the fundamental unit of the real quadratic field $\Q(\sqrt{p})$ for primes $p\equiv 1 \Mod{4}$ while providing a generalization of their conjecture. Using our generalization, we relate Fermat quotients of quadratic non-residues $\Mod{p}$ to sums of harmonic numbers. 
\end{abstract}

\section{Introduction}
In 1951, Ankeny, Artin, and Chowla \cite{Ankeny1951} derived four congruence relations for the class number of real quadratic fields $\Q(\sqrt{p})$ with $p$ a prime. In a later paper, \cite{Ankeny1952}, they published proofs of only three relations and perhaps inadvertently omitted the proof of the fourth relation. This gap was filled in by Carlitz \cite{Carlitz1953} who gave a proof but again omitted to write out several key steps of the proof. The missing steps involve the use of the $p$-adic logarithm, a profound idea introduced in \cite{Ankeny1952} and developed later by Iwasawa \cite{Washington}.  Looking back at the paper of Ankeny, Artin and Chowla, one cannot fail to see the birth of two fundamental concepts of number theory: one is the $p$-adic logarithm and the other is the use of the group ring to study cyclotomic fields, both of which had a transformative influence in the number theory of the twentieth century.

In this paper, we will amplify this idea giving simplified proofs of the results in \cite{Ankeny1951} and \cite{Ankeny1952}. At the same time, we extend these results and investigate a conjecture of Ankeny, Artin, and Chowla explained in the next section. 

\section{An extension of the Ankeny-Artin-Chowla congruence}
Let $p$ be a prime $\equiv 1 \Mod{4}$. The fundamental unit of $\Q(\sqrt{p})$ can be written as 
\begin{equation}\label{unit}
\epsilon = \frac{t+u\sqrt{p}}{2}
\end{equation}
for positive integers $t$ and $u$. Ankeny, Artin and Chowla \cite{Ankeny1951} stated the result
\begin{equation}\label{eq:1}
   \frac{2h u}{t} \equiv \frac{A+B}{p} \Mod{p} 
\end{equation}
where $h$ is the class number of $\Q(\sqrt{p})$, $A$ is the product of the quadratic residues $\Mod{p}$ lying in $[1,p]$ and $B$ is the product of the quadratic non-residues $\Mod{p}$ lying in $[1,p]$. This was one of the four results stated in \cite{Ankeny1951} and  three of these four results were proved in \cite{Ankeny1952}. Carlitz \cite{Carlitz1953} noted this gap and provided a proof but was vague in several key steps. The essential ingredient of the $p$-adic logarithm needed to derive (\ref{eq:1}) is not clearly enunciated in \cite{Carlitz1953} or \cite{Ankeny1952}. We will fill this gap in our discussion below. We will also use this occasion to give streamlined proofs of the results in \cite{Ankeny1951} and \cite{Ankeny1952}.

In their paper, Ankeny, Artin and Chowla conjecture that for primes $p\equiv 1\pmod 4$, we always have $p\nmid u$.  They verified their conjecture when $p\equiv 5 \pmod 8$ and $p<2000$.  Van der Poorten, te Riele and Williams \cite{PRW} \cite{PRW2} verified the conjecture for all primes $p< 2\cdot 10^{11}$.  The interest in this conjecture lies in the following observation.  In 1960, Ankeny and Chowla \cite{AC} proved that $h<p$.  For any given prime $p$, the quantities $A$, $B$ and $\epsilon$ are easily computed (using the continued fraction of $\sqrt{p}$ in the case of $\epsilon$) and so, 
(\ref{eq:1}) provides a congruence for $h\pmod p$.  But since $h<p$, the reduced residue gives the exact value of $h$ {\bf provided $u$ is not divisible by $p$}.

We briefly address remarks made by Ankeny and Chowla \cite{AC} regarding $h<p$.  In their paper, they wrote that this estimate is not well-known and cite a comment of Carlitz \cite{Carlitz1953} to this effect.  They then give a proof of the stronger estimate $h=O(\sqrt{p}).$  In a postscript to their paper, they provide another proof of the weaker estimate $h=O(p^{1/2+\epsilon})$ due to Mordell which uses the theory of reduced binary quadratic forms.  The proof given in \cite{AC} employs Dirichlet's class number formula though the authors never say so.
In fact, the formula on top of page 146 is pulled out of the ``hat of Dirichlet.''  

It is relatively painless to give a proof that $h=O(\sqrt{p})$ using Dirichlet's formula in the form
\begin{equation}\label{dirichlet}
 \frac{2h\log \epsilon}{\sqrt{|d_K|}}= \sum_{n=1}^\infty \frac{\chi(n)}{n} , 
 \end{equation}
where $\chi(n)=(d_K/n)$ is the Kronecker symbol, $d_K$ is the discriminant of $K={\mathbb Q}(\sqrt{p}).$
From (\ref{unit}), we see that $\epsilon \gg \sqrt{p}$ so that the left hand side of (\ref{dirichlet}) is
$$ \gg \frac{h\log p}{\sqrt{p}}. $$
On the other hand, the right hand side of (\ref{dirichlet}) can be re-written (by partial summation) as
$$ \sum_{n=1}^{p-1} \frac{\chi(n)}{n} + \int_p^\infty \frac{S(x)}{x^2} dx, $$
 where $S(x)=\sum_{n<x}\chi(n)$.  
As $S(x)=O(p)$, we see that the right hand side of (\ref{dirichlet}) is bounded by $O(\log p)$ and the claim is now immediate.  We remark that an identical proof shows that the class number $h(D)$ of ${\mathbb Q}(\sqrt{D})$ with $D$ a fundamental discriminant is also $O(\sqrt{|D|})$.  We will use this fact in a later section.

A natural question that arises is what result emerges in (\ref{eq:1}) if we replace $A$ and $B$ by complete sets $\{a_s\}_{s=1}^{(p-1)/2}$, $\{b_s\}_{s=1}^{(p-1)/2}$ of quadratic residues and non-residues (respectively) but not lying necessarily in $[1,p]$. We will assume that the $a_s, b_s$ are all positive. This question is easily answered as follows. 

Writing 
\begin{align*}
    a_s &= \langle a_s \rangle + p \left[ \frac{a_s}{p} \right]\\
    b_s &= \langle b_s \rangle + p \left[ \frac{b_s}{p} \right]
\end{align*}
where $\langle x \rangle$ denotes the reduced residue of $x\Mod{p}$. Then, 
\begin{align*}
    A &= \prod_{s=1}^{(p-1)/2} \left(a_s - p \left[ \frac{a_s}{p} \right] \right)\\
    B &= \prod_{s=1}^{(p-1)/2} \left(b_s - p \left[ \frac{b_s}{p} \right] \right).
\end{align*}
We easily see that 
\begin{align*}
    A & \equiv A^* - pA^* \sum_{s=1}^{(p-1)/2} \left[\frac{a_s}{p} \right] \frac{1}{a_s} \Mod{p^2}\\
    B & \equiv B^* - pB^* \sum_{s=1}^{(p-1)/2} \left[\frac{b_s}{p} \right] \frac{1}{b_s} \Mod{p^2}
\end{align*}
where 
\[
A^* = \prod_{s=1}^{(p-1)/2} a_s\,\,\,\,\, \text{ and }\,\,\,\,\, B^* = \prod_{s=1}^{(p-1)/2} b_s.
\]
Thus, 
\[
A+B \equiv A^* +B^* -p\left(A^* \sum_{s=1}^{(p-1)/2} \left[\frac{a_s}{p} \right] \frac{1}{a_s} + B^* \sum_{s=1}^{(p-1)/2} \left[\frac{b_s}{p} \right] \frac{1}{b_s} \right)\Mod{p^2}.
\]
Therefore, 
\[
\frac{A+B}{p} \equiv \frac{A^* +B^*}{p} -\left(A^* \sum_{s=1}^{(p-1)/2} \left[\frac{a_s}{p} \right] \frac{1}{a_s} + B^* \sum_{s=1}^{(p-1)/2} \left[\frac{b_s}{p} \right] \frac{1}{b_s} \right)\Mod{p}.
\]
This leads to the following variant of (\ref{eq:1}). 
\begin{thm}\label{thm:2.1}
    Let $p\equiv 1 \Mod{4}$ be a prime and $\{a_s\}, \{b_s\}$ be positive numbers representing a complete set of quadratic residues and non-residues $\Mod{p}$ respectively. Let $A^*$ and $B^*$ be the product of the $a_s$ and $b_s$ respectively. If $\epsilon = \frac{1}{2}(t+u\sqrt{p})$ is the fundamental unit of $\Q(\sqrt{p})$, then 
    \[
    \frac{A^* +B^*}{p} \equiv \frac{2hu}{t} + \left(A^* \sum_{s=1}^{(p-1)/2} \left[\frac{a_s}{p} \right] \frac{1}{a_s} + B^* \sum_{s=1}^{(p-1)/2} \left[\frac{b_s}{p} \right] \frac{1}{b_s} \right)\Mod{p}.
    \]
\end{thm}
\begin{rmk}
    If $1\leq a_s, b_s \leq p$, then the two sums in the congruence above vanish identically and we retrieve the Ankeny, Artin, and Chowla congruence (\ref{eq:1}). 
\end{rmk}

\section{Relation to Fermat quotients and harmonic numbers}
Let $R$ be a complete set of quadratic residues $\Mod{p}$ all lying in $[1, p-1]$ and $N$ a complete set of quadratic non-residues $\Mod{p}$ all lying in $[1, p-1]$. Then $A^* =A$ and $A\equiv -1 \Mod{p}$. Moreover, $B^*= B$ and $B\equiv 1 \Mod{p}$. Now we fix a quadratic non-residue $m\Mod{p}$ and define the two sequences 
\[
b_r = mr \,\,\, (r\in R) \,\,\,\,\, \text{ and }\,\,\,\,\,  a_n = mn \,\,\, (n\in N).
\]
In this set-up, Theorem \ref{thm:2.1} reads as 
\[
\frac{A+B^*}{p} \equiv \frac{2hu}{t} + B^*\sum_{r\in R}\left[\frac{mr}{p} \right] \frac{1}{mr} \Mod{p}
\]
and 
\[
\frac{A^*+B}{p} \equiv \frac{2hu}{t} + A^*\sum_{n\in N}\left[\frac{mn}{p} \right] \frac{1}{mn} \Mod{p}
\]
respectively. Observing that $B^* = m^{(p-1)/2}A$ and $A^* =m^{(p-1)/2}B$ respectively, and that $m^{(p-1)/2}\equiv -1 \Mod{p}$, we deduce
\begin{align}
    A \left(\frac{m^{\frac{p-1}{2}}+1}{p} \right) &\equiv \frac{2hu}{t} - A\sum_{r\in R}\left[\frac{mr}{p} \right] \frac{1}{mr} \Mod{p}\\
    B\left(\frac{m^{\frac{p-1}{2}}+1}{p} \right) &\equiv \frac{2hu}{t} - B\sum_{n\in N}\left[\frac{mn}{p} \right] \frac{1}{mn} \Mod{p}
\end{align}
This then leads to a refinement of \cite{Ankeny1951}:
\begin{thm}\label{thm:5.1}
 Suppose $p$ is a prime $\equiv 1 \Mod{4}$ and let $R$ a complete set of quadratic residues lying in $[1, p-1]$, $N$ a complete set of quadratic non-residues lying in $[1, p-1]$, 
\[
\epsilon = \frac{t+u\sqrt{p}}{2}
\]
be the fundamental unit of $\Q(\sqrt{p})$, and $h$ the class number of $\Q(\sqrt{p})$. Then for any quadratic non-reside $m\Mod{p}$,  
\begin{align}
    \frac{m^{p-1}-1}{p} &\equiv \frac{4hu}{t} +2\sum_{r\in R}\left[\frac{mr}{p} \right] \frac{1}{mr} \Mod{p}\\
   \frac{m^{p-1}-1}{p}  &\equiv -\frac{4hu}{t} +2\sum_{n\in N}\left[\frac{mn}{p} \right] \frac{1}{mn} \Mod{p}
\end{align}
\end{thm}
\begin{proof}
By the above discussion, the theorem follows at once by noting that $A\equiv -1\Mod{p}$ and $B\equiv 1 \Mod{p}$ and by multiplying equations (4) and (5) by $m^{(p-1)/2} -1 \equiv -2 \Mod{p}$. 
\end{proof}

\begin{defn}
For $p$ a prime number, the base $a$ Fermat quotient is the integer defined as 
\[
F(a) = \frac{a^{p-1}-1}{p}.
\]
\end{defn}
\noindent
Adding the two congruences in Theorem \ref{thm:5.1} we deduce:
\begin{cor}\label{cor:5.3}
  Suppose $p\equiv 1 \Mod{4}$ and $m$ is a quadratic non-residue $\Mod{p}$. Then,
  \[
mF(m) \equiv  2\sum_{k=1}^{p-1} \floor{\frac{mk}{p}} \frac{1}{k}
\]
where $F(m)$ is the Fermat quotient.
\end{cor}

\begin{rmk}
    We observe that this congruence relation for the Fermat quotient is unconditional of the Ankeny-Artin-Chowla conjecture. If $p\mid u$, then the terms involving the class number in Theorem \ref{thm:5.1} vanish. If $p\nmid u$, these same two terms cancel when added.  
\end{rmk}
\begin{defn}
The $k$-th harmonic number $H_k$ is defined as 
\[
H_k = \sum_{j=1}^k \frac{1}{j}
\]
where we understand that $H_0=0$. 
\end{defn}
\begin{lem}\label{lem:5.5}
  For any odd prime $p$, $H_{p-1} \equiv 0 \Mod{p}$. 
\end{lem}
\begin{proof}
    This is a simple matter of pairing up additive inverses. As $k$ runs through $\{1, \ldots, \frac{p-1}{2} \}$, $p-k$ will run through $\{p-1, \ldots, \frac{p+1}{2}\}$. Hence,
    \[
    H_{p-1} = \sum_{k=1}^{p-1} \frac{1}{k} \equiv \sum_{k=1}^{\frac{p-1}{2}}\left( \frac{1}{k} + \frac{1}{p-k}\right) \equiv  \sum_{k=1}^{\frac{p-1}{2}} \left(\frac{1}{k} - \frac{1}{k}\right) \equiv 0 \Mod{p}. 
    \]
\end{proof}
\noindent 
An immediate corollary of Lemma \ref{lem:5.5} is the following:
\begin{cor}\label{cor:5.6}
    Suppose $p$ is prime and $a<b$ are positive integers such that $a+b=p-1$, then 
    \[
    H_a \equiv H_b \Mod{p}.
    \]
\end{cor}
\begin{proof}
    By Lemma \ref{lem:5.5}, we have 
    \[
    0 \equiv H_{p-1} \equiv \sum_{j=1}^a \frac{1}{j} + \sum_{j=a+1}^{p-1} \frac{1}{j} \Mod{p}.
    \]
    By assumption, $a+1 = p-b$ and hence
    \[
    H_a \equiv - \sum_{j=p-b}^{p-1} \frac{1}{j} \Mod{p}.
    \]
    Making the change of variables, $j\to p-j$ we have, 
    \[
    H_a \equiv \sum_{j=1}^{b} \frac{1}{j} \equiv H_b \Mod{p}.
    \]    
\end{proof}
\noindent 
Setting $x_k= H_{k-1}$, Corollary \ref{cor:5.3} reads as
\[
mF(m) \equiv \sum_{k=1}^{p-1} \floor{\frac{mk}{p}} (x_{k+1}-x_k) \Mod{p}.
\]
Using summation by parts this becomes 
\[
mF(m) \equiv x_{p}\floor{\frac{m(p-1)}{p}} - \sum_{k=2}^{p-1} x_k \left(\floor{\frac{mk}{p}} - \floor{\frac{m(k-1)}{p}} \right) \Mod{p}. 
\]
The first term on the right hand side vanishes by Lemma \ref{lem:5.5}. Now in the summand, we simplify the difference of the floor functions using the following lemma:
\begin{lem}\label{Lem:3.4}
Suppose $p$ is prime and that $M$ is any positive coset representative of a positive reduced residue $m \Mod{p}$. Then,
\[
\floor{\frac{Mk}{p}} - \floor{\frac{M(k-1)}{p}} = \floor{\frac{M}{p}} + \floor{\frac{mk}{p}} - \floor{\frac{m(k-1)}{p}}
\]
for all $1\leq k \leq p-1$. In particular, if $M=m$ then 
\[
\floor{\frac{mk}{p}} - \floor{\frac{m(k-1)}{p}} = \begin{cases}
    1 & \text{if } k= \floor{\frac{p\ell}{m}} + 1 \text{ for some $\ell \in [0, m-1]$}\\
    0& \text{otherwise}.
\end{cases}
\]
\end{lem}
\begin{proof}
    Write $M= m + p \floor{\frac{M}{p}}$. Then,
    \[
    \floor{\frac{Mk}{p}} - \floor{\frac{M(k-1)}{p}} = \floor{\frac{mk}{p} + k\floor{\frac{M}{p}}} - \floor{\frac{m(k-1)}{p} + (k-1)\floor{\frac{M}{p}}}.
    \]
    As $\floor{x+n} = \floor{x}+ n$ for any positive integer $n$, we deduce the first part of the statement. 

    For the second statement, we note that $\floor{\frac{mk}{p}} - \floor{\frac{m(k-1)}{p}} \neq 0$ if and only if there exists some positive integer $\ell\leq  m-1$ such that  
    \[
    \frac{m(k-1)}{p} < \ell \leq \frac{mk}{p}.
    \]
    Rearranging this inequality, we have that
    \[
    k-1 < \frac{ p \ell}{m} \leq k
    \]
    and hence $k-1 = \floor{\frac{ p \ell}{m}}$. For these values of $k$, we seek to simplify: 
    \[
    \floor{\frac{m }{p}\left( \floor{\frac{ p \ell}{m}} +1 \right)} - \floor{\frac{m }{p}\floor{\frac{ p \ell}{m}}}. 
    \]
    Writing $p\ell = \theta + m \floor{\frac{p\ell}{m}}$ for $0<\theta < m$ we deduce that 
    \[
    \frac{m}{p} \floor{\frac{ p \ell}{m}} = \ell - \frac{\theta}{p}. 
    \]
    Using this relation we have:
    \[
\floor{\frac{m }{p}\left( \floor{\frac{ p \ell}{m}} +1 \right)} - \floor{\frac{m }{p}\floor{\frac{ p \ell}{m}}} = \floor{\ell- \frac{\theta}{p} + \frac{m}{p}} - \floor{\ell - \frac{\theta}{p}} = \floor{\frac{m-\theta}{p}} - \floor{-\frac{\theta}{p}}
    \]
Since $0< \theta< m$, the first floor function on the right is zero as $0<m-\theta <p$. The second term is $-1$ as $-1< -\theta/p <0$. The result then follows. 
\end{proof}
\begin{lem}\label{lem:3.5}
    Fix $p$ and some $q<p$. Then for all $1\leq k \leq q-1$ we have
    \[
    \floor{\frac{pk}{q}} + \floor{\frac{p(q-k)}{q}} = p-1.
    \]
\end{lem}
\begin{proof}
    Suppose $\floor{\frac{pk}{q}} = \ell$. Then we have that 
    \[
    \floor{\frac{p(q-k)}{q}} = p + \floor{\frac{-pk}{q}}
    \]
    and since $\floor{\frac{pk}{q}} = \ell$ we deduce that $\floor{\frac{-pk}{q}} = -\ell-1$. Hence, 
    \[
    \floor{\frac{pk}{q}} + \floor{\frac{p(q-k)}{q}} = \ell + (p - \ell -1) = p-1.
    \]
\end{proof}
Therefore, by Lemma \ref{Lem:3.4}, and partial summation we deduce:
\begin{thm}\label{thm:5.4}
If $p\equiv 1 \Mod{4}$ is a prime and $M$ is any positive coset representative of a  quadratic non-residue $m \Mod{p}$, then
\[
  -MF(M) \equiv \floor{\frac{M}{p}} + \sum_{j=1}^{m-1}H_{\floor{\frac{pj}{m}}}.
\]
\end{thm}
\begin{rmk}
    One can shorten the sum on the right in Theorem \ref{thm:5.4} by applying Corollary \ref{cor:5.6} and Lemma \ref{lem:3.5}. The length of the resulting sum depends only on the parity of $m$.  
\end{rmk}
\begin{rmk}
    Note that if $p\equiv 5 \Mod{8}$, so that $2$ is a quadratic non-residue, Theorem \ref{thm:5.4} yields
    \[
        -2F(2) \equiv H_{\floor{\frac{p}{2}}} \Mod{p}
    \]
    as the sum contains only a single term. Moreover, as $p$ is odd, $\floor{\frac{p}{2}} = \frac{p-1}{2}$. In effect, we recover Eisenstein's congruence 
    \[
    -2F(2) \equiv H_{\frac{p-1}{2}} \Mod{p}
    \]
    for all primes $p \equiv 5\Mod{8}$ \cite{Eisenstein1850}. In fact, we get a slight generalization of this result. Suppose $1<m<p$ is an odd quadratic non-residue with $p\equiv 1 \Mod{m}$. From the previous remark, Theorem \ref{thm:5.4}, and the fact that $\floor{\frac{pj}{m}} = \frac{(p-1)j}{m}$, we deduce that 
    \[
    -mF(m) \equiv 2\sum_{j=1}^{(m-1)/2} H_{\left(\frac{p-1}{m}\right) j} \Mod{p}.
    \]
\end{rmk}
\begin{rmk}
    The Fermat quotient $F(a) = \frac{a^{p-1}-1}{p}$ satisfies the ``logarithmic" functional equation 
    \[F(ab) = F(a)+F(b) \Mod{p}\] for $a$ and $b$ coprime to $p$.  Indeed, $a^{p-1} = 1 + pF(a)$ so that 
    \begin{align*}
       1 + pF(ab) = (ab)^{p-1} = a^{p-1}b^{p-1} = (1+ pF(a))(1+pF(b)) \equiv 1 + p(F(a)+F(b)) \Mod{p^2} 
    \end{align*}
     from which the desired additive congruence is immediate. It is important to note that the Fermat quotient is a function from $(\Z/p^2\Z)^\times \to \Z/p\Z$. 
\end{rmk}
\begin{thm}\label{thm:5.6}
    Suppose $p\equiv 1 \Mod{4}$ is prime. Then the Fermat quotient of any quadratic residue $\Mod{p}$ can be written as some linear combination of Fermat quotients of quadratic non-residues $\Mod{p}$.  Precisely, if $r$ is a quadratic residue $\Mod{p}$ such that $r\equiv \overline{a}\overline{b} \Mod{p^2}$ for quadratic non-residues $\overline{a},\overline{b}\Mod{p^2}$ such that $\overline{a}\equiv a \Mod{p}$ and $\overline{b}\equiv b \Mod{p}$ then,  
    \[
    -rF(r) \equiv \overline{b}\floor{\frac{\overline{a}}{p}} + \overline{a}\floor{\frac{\overline{b}}{p}} +  \overline{b}\sum_{j=1}^{a-1}H_{\floor{\frac{pj}{a}}} + \overline{a} \sum_{k=1}^{b-1}H_{\floor{\frac{pk}{b}}} \Mod{p}.
    \]
\end{thm}
\begin{proof}
This follows from Theorem \ref{thm:5.4} and the logarithmic property of the Fermat quotients. 
\end{proof}
\section{A generalized Ankeny-Artin-Chowla conjecture}

In a recent paper \cite{Yang-fu}, Yang and Fu formulated a generalization of the Ankeny-Artin-Chowla conjecture as follows.  Let $D$ be a positive integer which is not a perfect square.
Consider the set of all solutions of the Brahmagupta-Pell equation
$$ u^2 - Dv^2=1, $$
with $u,v$ natural numbers.  
Let $(u_1, v_1)$ be the least positive integer solution in this set and denote by $h(4D)$ the
class number of primitive binary quadratic forms of discriminant $4D$.  For $D$ odd, they conjecture that
$$v_1h(4D)\not\equiv 0 \pmod D. $$
We will refer to this as the generalized Ankeny-Artin-Chowla conjecture (GAAC).  Assuming this conjecture, Fu and Yang show that the equation
$$x^y + y^x = z^2 , \quad \min(x,y)>1, \quad (x,y)=1, \quad 2\not | xy, \quad x,y,z \in {\mathbb N}$$
has no solution. However, it appears that their conjecture has been made prematurely. Indeed, in \cite{Yu1998} there is a list of six positive squarefree $D< 10^{8}$ such that $v_1\equiv 0\Mod{D}$. Of the six counter examples only $D = 23\cdot 79, 3\cdot 69997$ and $41\cdot 79\cdot 541$ are odd. Each of these three $D$ is a counterexample to GACC. 

We will use an elementary sieve argument and basic algebraic number theory to show that GAAC is true for infinitely many discriminants $D$. We begin with a very simple set-theoretic sieve inequality:
\begin{lem}[The simple sieve]
    Let $S$ be a finite non-empty set and $I$ a finite indexing set. For each $i\in I$, we assign a set $A_i\subset S$. If $J\subseteq I$, then  
    \[
    \left| S\setminus \bigcup_{i\in I} A_i \right| \geq \left| S\setminus \bigcup_{j\in J} A_j \right| - \sum_{i\in I\setminus J} |A_i|.
    \]
\end{lem}
An application of the simple sieve plus an elementary counting argument yields:
\begin{lem} \label{lem:squarefree}
    The number of $n\in \N$ less than $x$ such that $n^2-1$ is square free is
    \[
    |\{n\leq x : n^2-1 \text{ is square free}\}|= Ax + O\left( \frac{x}{\log \log x}\right)
    \] 
    where 
    \[
    A= \prod_{p\leq \log \log x} \left(1-\frac{2}{p^2}\right).
    \]
\end{lem}
\begin{proof}
Let $x$ be sufficiently large.  For each prime $p\leq x$ we define:
\[
A_p = \{n\leq x: p^2\mid (n^2-1)\}.
\]
Then our goal is to estimate
\[
\left| \{n\leq x\}\setminus \bigcup_{p\leq x}A_p \right|.  
\]
A straight forward application of the simple sieve gives the lower bound
\[
\left| \{n\leq x\}\setminus \bigcup_{p\leq x}A_p \right| \geq \left| \{n\leq x\}\setminus \bigcup_{p\leq z}A_p \right| - \sum_{z\leq p \leq x} |A_p|.
\]
where $z$ is some parameter we will choose later. 

We make the observation that if $n\in A_p$, then $p^2\mid (n^2-1)$ and in particular, $p^2$ can only divide one of $n-1$ or $n+1$. Therefore, $A_p$ will be the number of $n\leq x$ that reduce to $1\Mod{p^2}$ or $-1\Mod{p^2}$. From this we deduce that 
\[
|A_p| = \frac{2x}{p^2} + O(1).
\]
Hence, 
\[
\sum_{z\leq p \leq x} |A_p| = 2x \sum_{z\leq p\leq x} \frac{1}{p^2} + O(\pi(x)).
\]
By the integral test, we see that the sum is bounded above by $1/z$ and therefore we can bound the sum as 
\[
\sum_{z\leq p \leq x} |A_p| \ll \frac{x}{z} + O(\pi(x)).
\]
Let $P_z = \prod_{p\leq z} p$. Then we have that
\[
|\{n\leq x : n\notin A_p \text{ for any $p\leq z$} \}| = \sum_{n\leq x} \sum_{\substack{d^2\mid (n^2-1)\\d\mid P_z}}\mu(d).
\]
Switching the order of summation we have
\[
\sum_{n\leq x} \sum_{\substack{d^2\mid (n^2-1)\\d\mid P_z}}\mu(d) = \sum_{d\mid P_z} \mu(d)\sum_{\substack{n\leq x \\n^2\equiv 1 (d^2)}} 1.
\]
The inner sum is  $\frac{x2^{\omega(d)}}{d^2} + O(2^{\omega(d)})$. Putting this into the above equality we have that
\[
\sum_{n\leq x} \sum_{\substack{d^2\mid (n^2-1)\\d\mid P_z}}\mu(d) = x\sum_{d\mid P_z}  \frac{\mu(d)2^{\omega(d)}}{d^2}+ O\left(\sum_{d\mid P_z} 2^{\omega(d)}\right).  
\]
Noting that $\mu(a)2^{\omega(a)}a^{-2}$ and $2^{\omega(a)}$ are multiplicative functions, we can write the two sums as products over all the primes less than $z$, i.e., 
\[
\sum_{n\leq x} \sum_{\substack{d^2\mid (n^2-1)\\d\mid P_z}}\mu(d) = x\prod_{p\leq z} \left(1 - \frac{2}{p^2} \right) + O\left(\prod_{p\leq z} (1+2)\right).
\]
The first product converges to a non-zero constant and the product in the error term can be estimated as $3^{\pi(z)}$. In all, we deduce that 
\[
|\{n\leq x : n^2-1 \text{ is square free}\}|= Ax + O\left(3^{\pi (z)}\right) + O\left(\frac{x}{z} + \pi (x)\right)
\]
for some absolute non-zero constant $A$. Choosing $z=\log\log x$, we can bound the first big-$O$ term by $O((\log x)^B) $ for some absolute constant $B$ and the second term by $O(\frac{x}{\log \log x})$. In all, 
\[
|\{n\leq x : n^2-1 \text{ is square free}\}|= Ax + O\left( \frac{x}{\log \log x}\right)
\]
\end{proof}
For a more detailed analysis of squarefree values of quadratic functions $f(n)=n^2+c$, we refer the reader to section two of \cite{Murty1999}.

\begin{thm}  If $D=n^2-1$ is squarefree, then $\varepsilon= n+\sqrt{n^2-1}$ is the fundamental unit of ${\mathbb Q}(\sqrt{D})$.  For such $D$, GAAC is true for $D$ sufficiently large.
\end{thm}
\begin{proof}  The fact that $\varepsilon$ is the fundamental unit of ${\mathbb Q}(\sqrt{D})$ is an exercise in \cite{Murty-Esmonde} (see Exercise 8.3.1 on page 115).  In the notation of conjecture GAAC, we have $v_1=1$.  As remarked earlier, $h(D)=O(\sqrt{|D|})$ and the assertion is now evident by Lemma \ref{lem:squarefree}.
    
\end{proof}
\section{The $p$-adic logarithm}
One can define the $p$-adic logarithm using the power series
\[
-\log(1-x) = \sum_{n=1}^\infty \frac{x^n}{n}.
\]
Let $v_p(n)$ be the largest power of $p$ dividing $n$. We denote by $|\cdot|_p$ the standard $p$-adic metric on $\Q$. Thus if $x=a/b$, $a,b\in \Z$, $\gcd(a,b)=1$ and $b\neq 0$, then 
\[
|x|_p = p^{v_p(b)-v_p(a)}.
\]
Since for a given prime $p$, $v_p(n) \leq \left[\log n / \log p \right]$, we see that if $|x|_p=\lambda <1$ then 
\[
-\log(1-x) = \sum_{n=1}^\infty \frac{x^n}{n}
\]
converges $p$-adically because 
\[
\left|\frac{x^n}{n} \right|_p \leq \lambda^n p^{v_p(n)} \to 0
\]
as $n\to \infty$. 

As usual, we denote by $\Q_p$ the completion of $\Q$ with respect to $|\cdot|_p$. Given $\Q_p$, we take its algebraic closure $\overline{\Q_p}$, we then complete it to obtain the $p$-adic analogue of the complex numbers denoted by $\C_p$. This field is algebraically closed (see for example, Proposition 5.2 of \cite{Washington}). 

The $p$-adic logarithm defined above for $|x|_p<1$ can now be extended to all of $\C_p^{\times}$ such that 
\begin{align*}
  \log xy = \log x + \log y 
\end{align*}
and $\log p =0$. To avoid confusion with the usual logarithm, we denote the $p$-adic logarithm as $\log_p$. 

To say that $|x|_p<1$ is equivalent to saying $v_p(x)>0$. If $v_p(x)\geq 1/(p-1)$, then in the series
\[
-\log(1-x) = \sum_{n=1}^\infty \frac{x^n}{n}, 
\]
we see that 
\[
v_p\left( \frac{x^{rp^v}}{p^v}\right) \geq \frac{rp^v - v(p-1)}{p-1} = \frac{r((p-1)+1)^v - v(p-1)}{p-1} \geq \frac{r{v\choose 2}(p-1)^2 }{p-1} \geq 1
\]
if $v\geq 2$, $r\geq 1$ or if $v=1$, $r\geq 2$ or if $v=0$ and $r\geq p-1$. Thus, we have 
\[
\left|\log(1-x)-\sum_{j=1}^p \frac{x^j}{j} \right|_p < \frac{1}{p}.
\]
In other words, 
\[
\log(1-x) = - \sum_{j=1}^p \frac{x^j}{j} \Mod{p}. 
\]
A simple extension of Wilson's theorem shows that for $1\leq j \leq p-1$,
\[
\frac{1}{p} {{p}\choose{j}} \equiv \frac{(-1)^{j-1}}{j} \Mod{p}.
\]
Therefore, 
\begin{align*}
   \log(1-x) & \equiv \sum_{j=1}^{p-1} \frac{1}{p}{{p}\choose{j}}(-1)^j x^j + \frac{1}{p}(-1)^p x^p \Mod{p}\\
   &\equiv \frac{1}{p} \left((1-x)^p-1 \right) \Mod{p}.
\end{align*}
This proves:
\begin{thm}[Ankeny, Artin, Chowla, 1952] \label{thm:4}
If $v_p(x-1)\geq 1/(p-1)$, then 
\[
\log_p(x) \equiv \frac{x^p-1}{p} \Mod{p}. 
\]
\end{thm}

We can relate this to Fermat quotients 
\[
F(a) = \frac{a^{p-1}-1}{p}
\]
for $(a,p)=1$ as follows. Writing $a^{p-1}= 1 + pF(a)$, we have
\[
\log_p(a) = \frac{1}{p-1}\log_p a^{p-1} = \frac{1}{p-1} \log_p(1+pF(a)).
\]
On the other hand, we have the $p$-adic series 
\[
\frac{1}{p-1} = -1 - p - p^2 -\cdots
\]
and
\[
-\log_p(1+x) = -x + \frac{x^2}{2} -\cdots 
\]
so that 
\begin{align*}
 \log_p(a) &= (-1 - p - p^2 -\cdots ) \left(pF(a) - \frac{p^2F(a)^2}{2}+ \cdots \right) \\
 &\equiv -pF(a) \Mod{p^2}.   
\end{align*}
Thus, we have:
\begin{thm}
    If $(a,p)=1$, then 
    \[
        F(a) \equiv -\frac{\log_p(a)}{p} \Mod{p}.
    \]
    In particular, if we have sets of integers $\{a_s\}_{s=1}^{(p-1)/2}$ and $\{b_s\}_{s=1}^{(p-1)/2}$ such that 
    \[
    \prod_{s=1}^{(p-1)/2} a_s = -1 + p \Omega 
    \]
    and 
    \[
    \prod_{s=1}^{(p-1)/2} b_s = 1 + p \Omega^*,
    \]
    then taking $p$-adic logarithms, we derive formula (3.2 ff) of \cite{Carlitz1953}:
    \begin{align*}
        \Omega &\equiv \sum_{s=1}^{(p-1)/2} F(a_s) \Mod{p}\\
        \Omega^* &\equiv \sum_{s=1}^{(p-1)/2} F(b_s) \Mod{p}
    \end{align*}
\end{thm}
\begin{rmk}
    It is this crucial discussion regarding the $p$-adic logarithm that is missing in \cite{Carlitz1953}. 
\end{rmk}

\section{The group ring, Gauss sums, and congruences}
Another important idea of \cite{Ankeny1952} is the use of the group ring attached to a Galois group $\Gamma$ and how it operates on the field elements. More precisely, let $K/\Q$ be a finite Galois extension with Galois group $\Gamma$. The group ring $\Q[\Gamma]$ consists of elements
\[
\sum_{g\in \Gamma} a_g g, \,\,\,\,\, a_g\in \Q
\]
and we define multiplication via
\[
\left(\sum_{g\in \Gamma} a_g g \right)\left(\sum_{h\in \Gamma }b_h h \right) = \sum_{z\in \Gamma} \left( \sum_{gh=z} a_gb_h\right)z.
\]
If $\Gamma$ is abelian, then $\Q[\Gamma]$ is also abelian.

One can extend the action of $\Gamma$ on the field elements $K$ to the action of $\Q[\Gamma]$
in the obvious way: for $\alpha \in K$,
\[
\left(\sum_{g\in \Gamma} a_g g \right)(\alpha) := \sum_{g\in \Gamma} a_g g(\alpha).
\] 
In the case of the cyclotomic field $\Q(\zeta_m)$ where $\zeta_m$ denotes a primitive $m$-th root of unity, the Galois group $\Gamma$, of $\Q(\zeta_m)/\Q$ is $\Gamma \cong \left(\Z/m\Z \right)^{\times}$, where the automorphism $\sigma_a$ corresponding to the coprime residue class $a\Mod{m}$ is given via 
\[
\sigma_a (\zeta_m) = \zeta_m^a.
\]
If $m=p$ is prime, we let $\zeta = \zeta_p$ and following \cite{Ankeny1952}, we denote by 
\[
G = \sum_{j=1}^{p-1} \left(\frac{j}{p} \right)\sigma_j
\]
the element in the group ring $\Q[(\Z/p\Z)^{\times}]$. Here $\left(\frac{\cdot}{p} \right)$ denotes the Legendre symbol. This element can be viewed as a ``Gauss sum". Precisely we have: 
\begin{lem}\label{lem:4.1}
    Let $p\equiv 1 \Mod{4}$ be a prime. Suppose $\zeta$ is a primitive $p$-th root of unity. Then for any $1\leq a\leq p-1$, 
    \[
        G(\zeta^a) \equiv \sqrt{p}\left(\frac{a}{p} \right)
    \]
    where $\left( \frac{\cdot}{p} \right)$ is the Legendre symbol. 
\end{lem}
\begin{proof}
    By definition we have
    \[
    G(\zeta^a) = \sum_{j=1}^{p-1} \left(\frac{j}{p} \right) \zeta^{aj}.
    \]
    Making the change of variables, $k=aj$ we have
    \[
    G(\zeta^a) = \sum_{k=1}^{p-1} \left(\frac{a^{-1}k}{p} \right) \zeta^k.
    \]
    By multiplicativity of the Legendre symbol and $(a/p) = (a^{-1}/p)$ we have 
    \[
    G(\zeta^a) = \left(\frac{a^{-1}}{p} \right) \tau
    \]
    where
    \[
    \tau = \sum_{k=1}^{p-1} \left(\frac{k}{p} \right)\zeta^k.
    \]
    By Corollary 4.6 of \cite{Washington}, $\tau$ evaluates to 
    \[
    \tau = \begin{cases}
        \sqrt{p}, & \text{if } p\equiv 1 \Mod{4}\\
        i\sqrt{p}, & \text{if } p\equiv 3 \Mod{4}.
    \end{cases}
    \]
    The result follows.  
\end{proof}

Given an element $\alpha\in K$ and $\beta = \sum_{g\in \Gamma} a_g g \in \Q[\Gamma]$, we will use the notation 
\[
\alpha^\beta := \prod_{g\in \Gamma} g(\alpha)^{a_g}.
\]
We remark that given $x\in \Q$, 
\[
x^G = 1 
\]
as the character sum vanishes, and $x$ is fixed by every group element.  

The Dirichlet class number formula can then be written using the group ring formalism as follows. If $p\equiv 1 \Mod{4}$,  $\epsilon$ is the fundamental unit of $\Q(\sqrt{p})$ and $h$ is the class number, then we have the familiar formula
\begin{equation}\label{eq:4.1}
2h \log \epsilon = - \sum_{j=1}^{p-1}\left(\frac{j}{p} \right) \log(1-\zeta^j).
\end{equation}
Here, we have crucially used that $p\equiv 1 \Mod{4}$ to evaluate the Gauss sum that appears in deriving (\ref{eq:4.1}). Exponentiating the expression in (\ref{eq:4.1}), we have that 
\[
\epsilon^{2h} = (1-\zeta  )^{-G}, 
\]
using our group ring formalism. As noted above, $x^G = 1$ for any $x\in \Q$, so 
\[
\epsilon^{2h} = ( 1 - \zeta )^{-G} = (\zeta -1 )^{-G}.
\]
If $n$ is a quadratic non-residue $\Mod{p}$, we make a change of variables in the sum defining $G$ by sending $j \mapsto nj$. Under this change of variables, we have
\[
G = \sum_{j=1}^{p-1} \left(\frac{nj}{p} \right) \sigma_{nj} = -\sum_{j=1}^{p-1} \left(\frac{j}{p} \right) \sigma_{nj}.
\]
Hence,
\[
(1-\zeta)^{-G} = (1-\zeta^n)^G
 \]
from which it follows that
\[
\epsilon^{2h} = (\zeta^n -1 ) ^{G}.
\]
Combining the two formulas for $\epsilon^{2h}$ gives
\begin{equation}\label{eq:4.2}
    \epsilon^{4h} = \left(\frac{\zeta^n -1 }{\zeta-1} \right)^G = \left(\frac{\zeta^n -1 }{n(\zeta-1)} \right)^G
\end{equation}
since $n^G=1$. 

We then have the following lemma, 
\begin{lem} \label{Lem:6}
    Let $\zeta$ be a primitive $p$-th root of unity for $p$ a rational prime. For a quadratic non-residue $\Mod{p}$, set 
    \[
    \alpha_j = \frac{\zeta^{nj}-1}{n(\zeta^j -1)}.
    \]
    Then $v_p(\alpha_j-1)\geq 1/(p-1)$.     
\end{lem}
\begin{proof}
    Writing $\omega = \zeta^j$, we have that 
    \[
    \alpha_j = \frac{1-\omega^n}{1-\omega} = \sum_{k=0}^{n-1}\omega^k.
    \]
    Since $\zeta\equiv 1\Mod{1-\zeta}$, we deduce that $\omega^k\equiv 1 \Mod{1-\zeta}$ for all $1\leq k\leq n-1$. Hence, $\alpha_j \equiv 1 \Mod{1-\zeta}$. Therefore,  $(1-\zeta)^{p-1}$ divides $(\alpha_j-1)^{p-1}$. We have that $(p)=(1-\zeta)^{p-1}$. Hence,
    \[
        (\alpha_j-1)^{p-1} \equiv 0 \Mod{p}.
    \]
    In particular, as $v_p(x)$ is multiplicative we have that 
    \[
        (p-1)v_p(\alpha_j-1) \geq 1.
    \]
    From which the desired conclusion follows.
\end{proof}
\noindent
Observing that
\[
\left(\frac{\zeta^n-1}{n(\zeta-1)}\right)^G = \prod_{j=1}^{p-1} \left(\frac{\zeta^{nj}-1}{n(\zeta^j-1)}\right)^{\left(\frac{j}{p} \right)}
\]
 and applying the $p$-adic logarithm to both sides of (\ref{eq:4.2}), we have that  
\[
4h\log_p \epsilon  \equiv \sum_{j=1}^{p-1}  \left(\frac{j}{p} \right)  \log_p   \left(\frac{\zeta^{nj}-1}{n(\zeta^j-1)}\right) \Mod{p}.
\]
By Theorem \ref{thm:4} and Lemma \ref{Lem:6} we deduce the congruence
\[
4h\log_p \epsilon \equiv  \frac{1}{p}\sum_{j=1}^p \left(\frac{j}{p} \right)\left( \left(\frac{\zeta^{nj}-1}{n(\zeta^j-1)} \right)^{p}-1\right)\equiv \frac{1}{n^p}G\left(\frac{\left(\frac{\zeta^{n}-1}{\zeta-1} \right)^{p}-n}{p} \right) \Mod{p}. 
\]
This is equivalent to 
\begin{equation}\label{Eq:groupRing}
4h \log_p(\epsilon) \equiv \frac{1}{n} G(f(\zeta)) \Mod{p}    
\end{equation}
where
\[
f(x) = \frac{1}{p} \left\{ \left(\frac{x^n-1}{x-1}\right)^p - \left(\frac{x^{np}-1}{x^p-1} \right) \right\}. 
\]
Noting that we can factor the two rational functions in the definition of $f(x)$, we deduce that $f(x)$ is a polynomial in $x$:
\begin{equation}\label{eq:10}
    f(x) = \frac{1}{p} \left( \left(\sum_{k=0}^{n-1}x^k\right)^p - \sum_{j=0}^{n-1} x^{kp} \right).
\end{equation}
Moreover, as $\zeta^p=1$, we see that the second sum in (\ref{eq:10}) equals $n$ when evaluated at $\zeta$.  Therefore,  $f(\zeta) = \frac{1}{p}\left(\left(\frac{\zeta^{n}-1}{\zeta-1} \right)^{p}-n \right)$. Following \cite{Ankeny1952} we wish to simplify $f(x) \Mod{p}$. In particular:
\begin{lem} \label{lem:7}
For $f(x)$ as defined above, 
\[
f(x) \equiv -\sum_{k=1}^{p-1} \sum_{j=0}^\infty \frac{1}{k} x^{nk + pj} + \sum_{k=1}^{p-1} \sum_{j=0}^{n-1} \frac{j+1}{k}x^{k+pj} \Mod{p}
\]
where the first sum is over $nk+pj< pn$. 
\end{lem}
\begin{proof}
We write $f(x)$ as 
\[
f(x)   = \frac{(x^n-1)^p(x^p-1) - (x^{pn}-1)(x-1)^p}{p} \cdot \frac{1}{(x-1)^p(x^p-1)}.
\]
Noting that $\frac{1}{p}{p\choose k} \equiv \frac{(-1)^{k-1}}{k} \Mod{p}$ for $1\leq k \leq p-1$ we have
\[
\frac{(x-1)^p}{p} \equiv \frac{x^p-1}{p} - \sum_{k=1}^{p-1} \frac{x^{p-k}}{k} \Mod{p}. 
\]
Writing $\ell= p-k$, the above expansion $f(x)$ can be written as 
\[
f(x) \equiv \sum_{\ell=1}^{p-1} \frac{x^{n\ell}}{\ell (x^p-1)} - \sum_{\ell =1 }^{p-1} \frac{x^\ell}{\ell} \cdot \frac{x^{pn}-1}{(x^p-1)^2} \Mod{p}.
\]
Writing
\begin{align*}
    \frac{1}{1-x^p} = \sum_{j=0}^{\infty} x^{pj}\,\,\,\,\, \text{ and }\,\,\,\,\, \frac{1}{(1-x^p)^2} = \sum_{j=0}^\infty (j+1)x^{pj},
\end{align*}
\[
f(x) \equiv -\sum_{\ell=1}^{p-1} \sum_{j=0}^\infty \frac{x^{n\ell + pj}}{\ell} + \sum_{\ell=1}^{p-1}\sum_{j=0}^\infty \frac{j+1}{\ell} x^{pj+\ell} - \sum_{\ell=1}^{p-1} \sum_{j=0}^\infty \frac{j+1}{\ell} x^{pj+pn+\ell} \Mod{p}.
\]
Changing the index of summation in the last sum above from $j \mapsto j-n$ we have
\[
f(x) \equiv -\sum_{\ell=1}^{p-1} \sum_{j=0}^\infty \frac{x^{n\ell + pj}}{\ell} + \sum_{\ell=1}^{p-1}\sum_{j=0}^\infty \frac{j+1}{\ell} x^{pj+\ell} - \sum_{\ell=1}^{p-1} \sum_{j=n}^\infty \frac{j+1}{\ell} x^{pj+\ell}  + \sum_{\ell = 1}^{p-1} \sum_{j=n}^{\infty} \frac{n}{\ell} x^{pj+\ell}\Mod{p}.
\]
The two middle sums cancel to give a finite sum. Next, we observe that the exponent in the last sum is over numbers greater than $pn$ and coprime to $p$. As such, we can write the exponent $u=nv + pj$ for some $v\in \{1, \ldots ,p-1\}$. As such, 
\[
u = nv +pj > pn \Rightarrow pj> n(p-v)>0.
\]
Hence, the exponent $u$ appears in the first sum as well. Making a change of variables in the first sum of $\ell \mapsto n\ell$ we see that the term corresponding to the exponent $u$ in the first sum has coefficient $-n/u$. Therefore, the last sum entirely vanishes. In effect, 
\[
f(x)\equiv -\sum_{\ell=1}^{p-1} \sum_{j=0}^\infty \frac{1}{\ell} x^{n\ell +pj} + \sum_{\ell=1}^{p-1}\sum_{j=0}^{n-1} \frac{j+1}{\ell} x^{pj+\ell} \Mod{p}
\]
where the first sum is over all $n\ell+pj < pn$. 
\end{proof}
Finally, we require one more computational tool.
\begin{lem}\label{lem:4.4}
    Suppose $p\equiv 1 \Mod{4}$. Then
    \[
    S = \sum_{k=1}^{p-1}\frac{1}{k} \left(\frac{k}{p} \right) \equiv 0 \Mod{p}.
    \]
\end{lem}
\begin{proof}
    We observe that 
    \[
    S \equiv \sum_{k=1}^{p-1} \frac{1}{p-k} \left(\frac{p-k}{p} \right) \equiv -\sum_{k=1}^{p-1}\frac{1}{k} \left(\frac{-k}{p}\right) \Mod{p}.
    \]
    Since $p\equiv 1 \Mod{4}$, $(-1/p)=1$. Adding the two representations of $S$, we deduce that
    \[
    2S \equiv 0 \Mod{p}.
    \]
    Hence, $S\equiv 0 \Mod{p}$.
\end{proof}

\begin{thm}[Ankeny, Artin, Chowla, 1952]
    Suppose $p\equiv 1 \Mod{4}$. Let $h$ denote the class number of $\Q(\sqrt{p})$ and $\epsilon = (t+u\sqrt{p})/2$ be the fundamental unit. Then 
    \[
    4h \frac{u}{t} \equiv -\frac{1}{n}\sum_{k=1}^{p-1} \frac{1}{k} \left[ \frac{nk}{p} \right] \left(\frac{k}{p} \right) \Mod{p}
    \]
    for any quadratic non-residue $n \Mod{p}$. 
\end{thm}
\begin{proof}
    Suppose $\zeta$ is a primitive $p$-th root of unity. Then by Lemma \ref{lem:7}, 
    \[
    f(\zeta) = -\sum_{\ell=1}^{p-1} \sum_{j=0}^\infty \frac{1}{\ell} \zeta^{n\ell} + \sum_{\ell=1}^{p-1} \frac{1}{\ell} (j+1)\zeta^\ell \Mod{p}
    \]
    where the first sum is over all $n\ell + pj < pn$. In particular, there are $1 + \left[ n\ell /p\right]$ such $j$ in the first sum. Changing the index of summation in the first sum from $\ell \mapsto p-\ell$ we have  
    \[
    f(\zeta) \equiv  \sum_{\ell=1}^{p-1} \frac{1+ \left[\frac{n\ell}{p} \right]}{\ell} \zeta^{-n\ell} + \frac{n(n+1)}{2}\sum_{\ell=1}^{p-1} \frac{1}{\ell} \zeta^\ell \Mod{p}.
    \]
    Then applying the element $G$ as defined in (\ref{Eq:groupRing}) and using linearity of $G$, we have
    \[
        G(f(\zeta)) \equiv \sum_{\ell=1}^{p-1} \frac{1+ \left[\frac{n\ell}{p} \right]}{\ell} G(\zeta^{-n\ell}) + \frac{n(n+1)}{2} \sum_{\ell=1}^{p-1} \frac{1}{\ell}G(\zeta^\ell) \Mod{p}.
    \]
  By Lemma \ref{lem:4.1}, this reduces to 
  \[
    G(f(\zeta)) \equiv \sqrt{p}\sum_{\ell=1}^{p-1} \left(\frac{-n\ell}{p} \right)\frac{1+ \left[\frac{n\ell}{p} \right]}{\ell}  + \frac{n(n+1)}{2} \sqrt{p}\sum_{\ell=1}^{p-1} \frac{1}{\ell} \left(\frac{\ell}{p} \right) \Mod{p}
  \]
    Expanding the first sum out, applying Lemma \ref{lem:4.4}, and noting that $-n$ is a quadratic non-residue $\Mod{p}$, we deduce that 
    \[
    G(f(\zeta)) \equiv - \sqrt{p} \sum_{\ell=1}^{p-1} \frac{1}{\ell} \left[\frac{n\ell}{p} \right] \left(\frac{\ell}{p} \right) \Mod{p}. 
    \]
    Therefore, 
    \[
    4h\log_p(\epsilon) \equiv - \sqrt{p} \sum_{\ell=1}^{p-1} \frac{1}{\ell} \left[\frac{n\ell}{p} \right] \left(\frac{\ell}{p} \right)\Mod{p}. 
    \]
    Noting that
    \[
    \log_p(\epsilon) \equiv \frac{u\sqrt{p}}{t} \Mod{p},
    \]
    we conclude the result. 
\end{proof}

\section{Concluding remarks}
We have presented a self-contained treatment of some of the results in \cite{Ankeny1952} and \cite{Carlitz1953}. We generalized the AAC conjecture in Theorem \ref{thm:2.1} and thus, related the AAC conjecture to a congruence of Fermat quotients. This generalization allowed us to obtain a generalization of a result of Eisenstein relating Fermat quotients to sums of harmonic numbers. \\

The function field analogue of the Ankeny-Artin-Chowla conjecture has been studied by Yu and Yu \cite{Yu1998}.
The case $p\equiv 3 \pmod 4 $ has been studied by Mordell \cite{mordell1} and \cite{mordell2} and he
made a similar conjecture.  In all cases, we have an interesting connection to Bernoulli numbers.  In particular, if there are infinitely many regular primes (that is, primes $p$ that do not divide the class number of the $p$-th cyclotomic field), then the Ankeny-Artin-Chowla and the Mordell conjectures are true for those primes $p$.
It is unknown at present whether there are infinitely many regular primes though the conjecture is that there is a positive density of such primes.  
A relevant survey article by Slavutskii \cite{Slavutskii2004} is worth a careful study.

\renewcommand\refname{References}
\bibliography{References}
\end{document}